\documentclass[a4paper,10pt]{article}

\usepackage{amsmath}
\usepackage{amsthm}
\usepackage{amsfonts}
\usepackage{textcomp}     % access \textquotesingle
\usepackage{hyperref}
\usepackage{relsize}
\usepackage{sectsty}
\usepackage{wasysym}

\parskip\smallskipamount
\newtheorem*{theorem*}{Theorem}
\newtheorem{theorem}{Theorem}
\newtheorem{corollary}[theorem]{Corollary}
\newtheorem{proposition}[theorem]{Proposition}

\newtheorem{lemma}[theorem]{Lemma}

\newtheorem{remark}{Remark}

\newcommand{\apost}{\textquotesingle}
\newcommand{\pr}{\mathbb{P}}
\newcommand{\E}{\mathbb{E}}
\newcommand{\R}{\mathbb{R}}

\newcommand{\Z}{\mathbb{Z}}

\newcommand{\Na}{\mathbb{N}}

\newcommand{\Ze}{\textbf{Z}}
\newcommand{\Se}{\textbf{S}}

\newcommand{\B}{\mathcal{B}}

\DeclareMathOperator*{\sech}{sech}
\newcommand{\expo}{\textup{exp}}
\newcommand{\dis}{\displaystyle}
\newcommand{\til}{\widetilde}
\newcommand{\wdh}{\widehat}

\newcommand{\hs}{\hspace{2mm}}

\title{The several dimensional gambler's ruin problem}
\author{\textbf{Achillefs Tzioufas}\footnote{\small{\texttt{tzioufas@ime.usp.br}}}}
\begin{document}
\maketitle
\vspace{-6mm}

\begin{abstract}     
\noindent We consider the simple random walk on the $N$-dimensional integer lattice from the perspective of evaluating asymptotically the duration of play in the multidimensional gambler\apost s ruin problem. We show that, under suitable rescalings, all $p$-moments of exit-times from balls in the $L$-infinity metric, and all $p$-moments of partial-maxima values in this metric, possess associated asymptotic limit expressions, admitting two representations each. We derive for this purpose multidimensional refinements of the corresponding two-folded extension of Erd\H os-Kac theorem, which we revisit to this end. 
We show in particular a simplifying proof approach, which relies on an application of the optional stopping theorem, and yields the corresponding first-passage times asymptotics in parallel. We observe a direct manner of proof of the relation among the two limit expressions by Brownian motion scaling. 
We indicate in a manner intended to be brief and comprehensive other known proof approaches for the purposes of comparison and completeness. 
\vspace{5mm}


\vspace{5mm}


\end{abstract}

\section{ \Large  Introduction} 
 
\subsection{Motivation: fair games of chance}

Originating from correspondence of Blaise Pascal and Pierre de Fermat in 1656, the `gambler\apost s ruin problem' regards the game of two players engaging in a series of independent and identical bets up until one of them goes bankrupt, viz.\ ruined.\footnote{The first formulation of the gambler\apost s ruin problem had always been credited to work of Huygens in 1657, only because his correspondence, which mentions his source, was not published until 1888. For more on the historical background to this problem and its time-limited extension, we refer for instance  to the notes in [$\mathsection$ 7.5, Ethier \cite{E10}] and the references therein.} The general `gambler\apost s ruin formula', which regards the chances of each player winning, was shown by Abraham De Moivre in 1712. The solution to the problem of the `duration of play', which regards a `time-limited extension' of the said formula, also dates back to 1712 and is due to De Moivre.\footnote{A derivation of this formula may be found in [Feller \cite{Fl68}, Chpt.\ XIV, $\mathsection$ 5], where the technique of expanding rational functions in partial fractions is employed.} Different formulae for this were obtained afterward by Montmort, Nicolaus Bernoulli, as well as Joseph-Louis Lagrange.\footnote{see (7.38) and (7.164-5) respectively in \cite{E10}.} Regarding the fair bets case it is a celebrated result that the expected value of the duration of play equals the product of the initial fortunes of the players.\footnote{A derivation of this via Doob's Optional-Stopping Theorem may be found in any standard textbook in probability dealing with martingales.} Some of the original motivation for this work may be sought into study of generalizations of this formula in the higher-dimensional setting, which we describe next. 

The `fair gamblers\apost\mbox{ }ruin problem' may be cast as a simple betting game in the following settings, which extend the classical one in a natural way. In one of the interpretations, two players are both in possession of initial fortunes in a number of more than one currencies.  At every round of this game a fair bet takes place. The winner of the bet receives a payoff which amounts to one monetary unit of a currency chosen independently in an even way among all currencies. The game continues in independent rounds and is over as soon as either one of the players runs out of any currency.   

This problem may also be interpreted in the setting of a gambling competition between two teams with the same number of players. In this interpretation, opponents are matched into pairs in the outset. Every round of this game then consists of choosing evenly one of the pairs. The chosen pair of players then bets on a fair game and the winner receives one monetary unit from its opponent. The game carries on in this fashion for as long as none of the players of either team is bankrupt. Note that, since all bets are fair and the probability of winning for either player in the former interpretation, or for either team in the latter one, is obviously the same, the sole quantity of interest in the study of this game is the duration of its play. 

Nonetheless the aforementioned elegant and simple expression for the expected value of the duration of play in the one-dimensional fair gambler\apost s ruin problem, concerning the more general settings described, in the words of Orr and Zeilberger \cite{OZ94} `no closed-form solution of this problem is known to exist, and probably none does exist'.\footnote{Something which is also in accord with the, out of his own experience, belief of the author.} Prior to presenting our results regarding asymptotics for all moments of the duration of play in the fair gamblers\apost\mbox{} ruin problem, a brief review of the long mathematical history that the intimately associated `absorption problem' enjoys is given. 

\subsection{The Absorption Problem}\label{sub12}

The `simple one-dimensional random walk' is the celebrated discrete-time stochastic process comprising the sequence of successive partial sums of a sequence of independent uniformly distributed $\{-1,1\}$-valued random variables, which may be thought of as modeling the motion of a particle on $\Z$ jumping at every instance of time to either one of its nearest-neighbors sites according to the outcomes of this sequence. The `absorption problem' regards the asymptotic law of the times that the modulus of this walk attains new maxima values, under appropriate rescaling; note also that these times correspond to the particle\apost s `exit-times' from symmetric intervals about the origin.\footnote{Regarding the problem\apost s nomenclature, we note that it derives from the equivalent perspective of the asymptotic law of the duration of the motion of the particle in finite intervals with absorbing endpoints, as their length tends to infinity.} Note that, in principle this asymptotic law may be derived from any of the aforementioned solutions to the problem of the duration of play.

Spitzer [Chpt.\ V, \cite{Sp76}] provides with an original review to the absorption problem, pointing out to a list of earlier treatments for completeness\footnote{cf.\ [footnote 1, p.\ 237, \cite{Sp76}}. Regarding arbitrary zero-mean (positive and finite variance) increment-distributions random walks, the rigorous solution to this problem is due to Erd\H os and Kac [Theorem II, \cite{EK46}], although the asymptotic distribution was indeed already known in the beginnings of the previous century to Bachelier who worked in the context of Brownian motion. The crux of the Erd\H os and Kac solution to the general problem is that, for all zero-mean increment-distributions, the limiting distribution must be identical, up to a scaling constant which only depends on the variance of the increment-distribution. The study of random walks is intimately associated to the celebrated continuous-time stochastic process referred to as Brownian motion, a.k.a.\ the Wiener process. Indeed, the key idea of the Erd\H os and Kac solution paved the way to establishing the deep connection amongst random walks and Brownian motion, known as the `invariance principle', a.k.a.\ the `functional central limit theorem', or `Donsker\apost s theorem', due to Donsker \cite{Dn51}.\footnote{Note that the equivalence among the before-mentioned absorption problems for arbitrary zero-mean increment-distribution random walks and for the Brownian motion may be justified by an instance of this theorem, see [(\ref{TAe}), Theorem \ref{EK}] below for a precise formulation.}  

The mathematically rigorous theory of stochastic processes in continuous-time was initiated with the seminal work by Kolmogorov \cite{Kl31}. A rigorous solution of the associated absorption problem for Brownian motion was made available via the celebrated `L\'evy\apost s triple law', which regards the joint probability distribution of the Brownian motion at fixed times together with its running minimum and maximum, and is derived by L\'evy \cite{Lv48}.\footnote{For a justification of this claim we refer to [footnote 1, p.\ 169, Schilling and Partzsch \cite{SR14}]; cf.\ also with Remark \ref{levrem} below for pointers to the literature regarding its proof.} The closely related distribution of the running-maxima of the modulus of the Brownian Bridge, which corresponds to the distribution of the error term in the Kolmogorov-Smirnov statistic, was derived by Kolmogorov \cite{Kl33}.\footnote{For a derivation of this we refer to [Feller \cite{Fl70}, Chpt.\ I, $\mathsection$ 12].}

A different, elementary approach for deriving a time-limited gambler\apost s ruin expression in the fair bets case is shown in [$\mathsection$ 21, \cite{Sp76}].\footnote{The technique there relies on generalizations of spectral representations for transition matrices and orthogonal polynomials, and is different from those mentioned above; see, for instance, [$\mathsection$ 5, Chpt.\ 10, Karlin and Taylor \cite{KT81}] for an introduction to this approach.} This expression allows by means of a limiting procedure the derivation of a solution to the absorption problem, cf.\ [$\mathsection$ 21, Proposition 5, \cite{Sp76}], which by appeal to the invariance principle yields the corresponding result in the general case of arbitrary zero-mean increment-distribution random walks. Nevertheless, relying on the method of moments, an extension of this approach to cover the general case is in addition offered in [$\mathsection \mathsection$ 22, 23, \cite{Sp76}]; this approach does not only rely on technically challenging and novel methods of intrinsic interest, but also provides with various additional results, interesting in their own right.

\subsection{The multidimensional case}
The simple $N$-dimensional random walk is the basic discrete-time stochastic process modeling the motion of a particle on $\Z^{N}$ that is composed by unit-length displacements in each of the cardinal directions with equal probability, independently at every instance of time.\footnote{More formally this may be put as considering the successive partial-sums of uniformly distributed $\mathcal{B}_{N}\cup -\mathcal{B}_{N}$-valued random vectors, where $\mathcal{B}_{N}$ is the standard orthonormal basis of unit vectors of $\Z^{N}$ and $-\mathcal{B}_{N}$ are their negative-signed counterparts; cf.\ with (\ref{eq1}) for precise definition.} The analysis and understanding of random walks has been the epicenter of the theory of probability since the very inception of the subject, and is therefore one of its most exhaustively studied topics. For the classic reference devoted to the random walk, we refer to Spitzer \cite{Sp76}; for more recent accounts, we refer to Lawler and Limic \cite{LL10} and R\'ev\'esz \cite{R90}.   For introductory treatments devoted to the random walk, we refer to Lawler \cite{Lw10}. For  resources on background probability material, including excellent accounts on basic random walk theory, we refer to Billingsley \cite{Bl68}, Breiman \cite{Br92}, and Feller \cite{Fl68, Fl70}, whose everlasting influences cannot be overestimated. For more recent treatments in this regard, we refer to Durrett \cite{Dr10}, Gut \cite{G12, G09}, Kallenberg \cite{Kll97}, Stroock \cite{St05}, and Williams \cite{W91}. For accounts laying emphasis to Brownian Motion, we refer to M\"orters and Peres \cite{MP10}, Revuz and Yor \cite{RY94}, Rogers and Williams \cite{RW93}, and Schilling and Partzsch \cite{SR14}.
% Regarding the simple $N$-dimensional random walk,

As in the 1-dimensional case, asymptotics for the duration of play in the fair gamblers\apost\mbox{} ruin problem are associated to multidimensional absorption problems or, in other words, to asymptotics for the corresponding absorption times in many dimensions. Further, it is easy to see that, in the case of equal initial fortunes, the duration of play in the fair gamblers\apost\mbox{} ruin problem corresponds to exit-times of the $N$-dimensional walk from hypercubes, which is, $L^{\infty}$-balls about the origin. Note that the expected values of exit-times from $L^2$-balls of radius $r$ about the origin scaled with $r^{2}$ tends to 1, as $r \rightarrow \infty$, regardless of $N$.\footnote{One can easily show that indeed this expected value lies in $[r^{2}, (r+1)^{2})$ by a straightforward extension of the application of Doob\apost\mbox{} s Optional-Stopping Theorem we mentioned in footnote 3; cf., for instance, with [$\mathsection$ 1.4.2, \cite{Lw10}] for an explicit computation.}

Theorem \ref{momeN} is our main result and is formally stated below in Section \ref{S3}. Its first part provides asymptotics for all $p$-moments of the duration of play in $N$-dimensions, which in contrast turn out to depend on the dimension $N$. In addition, the second part of Theorem \ref{momeN} provides with asymptotic limits for all $p$-moments of partial-maxima values of the $N$-dimensional walk in the $L^{\infty}$-metric 
 under appropriate rescaling. These two asymptotic limits turn out to be associated, which is an aftereffect of the fact that the two random sequences are inverses of one another, in that exit times correspond to instances that new partial-maxima values are attained in this metric. In this sense Theorem  \ref{momeN} makes the resultant limit interconnection precise. The technique of proof of Theorem \ref{momeN} is elementary probabilistic with a perspective to studies of the random walk in conjunction with the Brownian motion, and indeed Proposition \ref{corBM} regarding the latter is key to its proof.

We note that the former-mentioned asymptotics in the first part of Theorem \ref{momeN} extend those derived by Kmet and Petkov\v{s}ek \cite{KP02}, that deal with expected values (1-moments) by means of discrete Fourier methods for solving the associated Poisson partial differential equation. The expressions we derive in this case are also contrasted and found to be simpler alternatives to the general-dimension asymptotics in \cite{KP02}, mostly in that our formulae involve $1$-fold instead of $N$-fold sums (cf.\ Remark \ref{kpet}). Our main results we made mention of are formally presented in Section \ref{S3} below. 

In addition, in Section \ref{EKsec} we revisit the version of the Erd\H os and Kac theorem, stated in Theorem \ref{EK} below. Subsection \ref{subEK} comprises a related observation and a simple consequence of this theorem, which we invoke later. 
In Subsection \ref{app} we show an elementary path, which to the limits of our knowledge is not pursued elsewhere, to prove this Theorem for completeness,  whereas in Subsection \ref{routeM} we give a miscellany of other available proof approaches to this theorem from the literature. 

\section{Statement of Results}\label{S3}

\subsection{The Simple Random Walk on $\Z^{N}$}
The simple random walk $(\Ze_{t}: t \geq 0)$ on the $N$-dimensional integer lattice $\Z^{N}$, may be defined via a collection of independent random variables $(\boldsymbol{\omega}_{s}: s \geq 1)$ with identical distribution which is given by  
\[
\pr(\boldsymbol{\omega}_{s} = \boldsymbol{w}) = \frac{1}{2N}, \hspace{3mm} \mbox{for } \mbox{ }  \boldsymbol{w} = \pm {\bf e}_{i}, \mbox{ }  i =1, \dots, N,
\] 
where ${\bf e}_{i}$ is the $N$-dimensional vector whose $i$th component equals $1$ and others equal $0$ \footnote{so that $\mathcal{B}_{N} = ({\bf e}_{i})_{i = 1, \dots, N}$, mentioned above, is the standard orthonormal basis of $\R^{N}$.}. To define the process, we then let 
\begin{equation}\label{eq1}
\dis{\Ze_{t} = \sum_{s=1}^{t}\boldsymbol{\omega}_{s}}. 
\end{equation}

Furthermore, we may define $(\Ze_{t})$ as the time-homogeneous Markov chain with state-space $\Z^{N}$ and transition probabilities given by, 
\begin{equation}\label{eq2}
\dis{ \pr(\Ze_{t} = {\bf z}|\mbox{ } \Ze_{t-1} = {\bf y}) = \frac{1}{2N}}, 
\end{equation}
for ${\bf z} - {\bf y} = \{\pm {\bf e}_{i}, i =1, \dots, N\}$, and $\Ze_{0} = {\bf 0}$.  

\subsection{A preparatory statement}
We give the following auxiliary statement we require  for stating our main result below. In order to see its connection with $\Ze_{t}$, note that the covariance matrix of this process is $\Sigma:= \E({\bf Z}_{1}{\bf Z}_{1}^{T}) = \frac{1}{N}I$, where $I$ denotes the identity matrix.

\begin{proposition}\label{corBM}
Let $\wdh{\bf{W}}_{s}$ be $N$-dimensional Brownian motion with covariance matrix $\frac{1}{N}I$. Let also $\wdh{T}_{N} = \inf\{s: {\bf \wdh{W}}_{s} \not\in {\bf B}_{1} \}$, where $\dis{{\bf B}_{1} = \{ {\bf x} \in \R^{N}: |{ \bf x}| \leq 1\}}$ and, in addition, let  $\wdh{M}_{N} = \sup_{0 \leq s \leq 1}|{\bf \wdh{W}}_{s}|. $
We have that
\begin{equation}\label{Ftil}
F_{N}(t) := \pr(\wdh{T}_{N}<t) = 1- \left(H\left(\frac{t}{N^{2}}\right)\right)^{N},
\end{equation}
and that 
\begin{equation}\label{Gtil}
G_{N}(x) := \pr(\wdh{M}_{N} <x) = \left(H\left(\frac{1}{N^{2}x^{2}}\right)\right)^{N},
\end{equation}  
where 
\begin{equation}\label{eq00}
H(y) = \frac{4}{\pi} \sum_{n \geq 0} \frac{(-1)^{n}}{2n+1} \cdot \textup{exp}\Big(-\frac{\pi^{2}}{8} (2n+1)^{2} y \Big), \mbox{ } y >0. 
\end{equation}
\end{proposition}

\begin{remark}
\textup{The function $H(\cdot)$ can be expressed in an alternative manner, as follows.
\begin{equation}\label{eq01}
H(y) = \int_{-1/\sqrt{y}}^{1/\sqrt{y}} \sum_{k = -\infty}^{\infty}(-1)^{k} \exp\left(-\frac{1}{2} \left(t+\frac{2k}{\sqrt{y}}\right)^{2}\right) dt.
\end{equation}
For practical purposes, (\ref{eq00}) is more useful for large values of $y$, whereas (\ref{eq01}) converges faster only for small values of $y$.\footnote{cf.\ with, for instance, [Remark 1, p.\ 21, R\'ev\'esz, \cite{R90}].}. For more on (\ref{eq00}), (\ref{eq01}) and their equivalence, see Subsection \ref{routeM} below. Finally, for ease of reference below, we note here that from (\ref{Ftil}) we clearly have that, for all $p \geq 1$, 
\begin{equation}\label{computmom}
\pr(\wdh{T}_{N}^{p} \geq t) = 1- F_{N}(t^{1/p}) = \left(H\left(\frac{t^{1/p}}{N^{2}}\right)\right)^{N}.
\end{equation} 
}
\end{remark}

\subsection{Main result}
In order to state our main result next, we let 
\begin{equation}\label{defTtil}
\til{T}_{N, r} = \min\{t: \Ze_{t} \not\in \B_{r}\}, 
\end{equation}
where $\dis{\B_{r} = \{ {\bf z} \in \Z^{N}: |{ \bf z}| \leq r\}}$, and $|\cdot|$ denotes the $L$-infinity norm\footnote{Recall that $|{\bf z}| = \max\{|z_{1}|, |z_{2}|, \dots, |z_{N}|\}$, ${ \bf z} = (z_{1}, \dots, z_{N})$, so that $\B_{r}$ is the hypercube with vertices $(\pm r, \dots, \pm r)$, a.k.a.\ the Moore neighborhood range $r \geq 1$.}. 
Furthermore, we let 
\begin{equation}\label{defMtil}
\til{M}_{N, t}= \max_{1 \leq s \leq t}|\Ze_{s}|. 
\end{equation}

\begin{theorem}\label{momeN} 
We have that
\begin{equation}\label{eqA}
\frac{1}{r^{2p}} \E \til{T}_{N, r}^{p}  \rightarrow \E \wdh{T}_{N}^{p} \hs \mbox{ as } r \rightarrow \infty, 
\end{equation} 
and further, that 
\begin{equation}\label{eqB}
\frac{1}{t^{p/2}}\E \til{M}_{N, t}^{p} \rightarrow  \E \wdh{M}_{N}^{p}  \hs \mbox{ as } t \rightarrow \infty,
\end{equation}
$N, p\geq 1$, where $\wdh{T}_{N}$ and $\wdh{M}_{N}$ are as in Proposition \ref{corBM}. 
\end{theorem}

\begin{remark}\label{kpet}
\textup{The methods in \cite{KP02} yield an explicit formula for the limit considered in (\ref{eqA}) in the case $p=1$, see [\cite{KP02}, $\mathsection$ 5, Theorem 2], and further, yield an estimate for the associated convergence rate for $N=2$, see [\cite{KP02}, $\mathsection$ 5, Theorem 1]. Their formula in our notation for the limit for $p=1$ in (\ref{eqA}) equals 
\begin{equation*}\label{eqpet}
N \Big( 1 - \frac{2^{2N+1}}{\pi^{N+1}} \sum_{k_{1}, k_{2}, ..., k_{N} \geq 0} \frac{{(-1)}^{\sum_{j=1}^{N-1}k_{j}} \prod_{j=1}^{N-1}(1/(2k_{j}+1))\sum_{j=1}^{N-1}(1/(2k_{j}+1)^{2})}{\cosh(\frac{\pi}{2}) \sqrt{ \sum_{j=1}^{N-1}(2k_{j}+1)^{2}}} \Big). 
\end{equation*}
On the other hand, [(\ref{eqA}), Theorem \ref{momeN}] together with (\ref{computmom}), by standard moment expressions for positive random variables, and a change of variables, gives that the limit for any $p\geq 1$ in (\ref{eqA}) equals}
\begin{equation*}
p  \int_{0}^{\infty} t^{p-1} \left(H\left(\frac{t}{N^{2}}\right)\right)^{N} dt,
\end{equation*}
\noindent \textup{where $H$ may be as in (\ref{eq00}), or as in (\ref{eq01}).}
\end{remark}

\subsection{A Preliminary}
A key preliminary result will be the celebrated Erd\H os and Kac \cite{EK46} theorem, an extended two-folded version that is apt for our purposes is stated next. 

\begin{theorem}\label{EK}
Let $S_{t} = \sum_{s=1}^{t} \xi_{s}$, where $(\xi_{i}; i \geq 1)$ are i.i.d.\ random variables such that $\E(\xi_{i}) = 0$ and that $\E(\xi^{2}_{i}) = 1$. Let $\tau_{b} = \inf\{ t: |S_{t}| \geq b\}$ and also let  $m_{t} = \max\{|S_{i}|: i \leq t\}$. Further, let $(W(t): t \in \R_{+})$ be standard linear Brownian motion and let $T = \inf\{t: |W(t)| = 1\}$ and also let $M = \sup_{t \in [0,1]}|W(t)|$.
We have that
\begin{equation}\label{TAe} 
\frac{\tau_{b}}{b^{2}} \xrightarrow{d} T \hs \mbox{ as } b \rightarrow \infty,
\end{equation}
and that,
\begin{equation}\label{Be}
\frac{1}{\sqrt{t}} m_{t} \xrightarrow{d} M \hs \mbox{ as } t \rightarrow \infty. 
\end{equation}
Further, if $\Phi(t) := \pr(T<t)$ and $\Gamma(x) := \pr(M <x)$, then we have that
\begin{equation}\label{PHI}
\Phi(t) = 1-  H(t),
\end{equation}
and that 
\begin{equation}\label{GAM}
\Gamma(x) = H\left(\frac{1}{x^{2}}\right), 
\end{equation}
where $H(\cdot)$ is as in (\ref{eq00}), or (\ref{eq01}).
\end{theorem}

\subsection{Outline of the proofs}

The remainder comprises Sections \ref{EKsec} and \ref{31p}. An outline of their contents and of the manner in which we organize these sections is given as follows.

In Section \ref{31p} we give the proofs of Theorem \ref{momeN} and Proposition \ref{corBM} stated above. The proof of Theorem \ref{momeN} relies on first showing convergence in distribution analogues of [(\ref{eqA}), Theorem \ref{momeN}] and of [(\ref{eqB}), Theorem \ref{momeN}], derived from Proposition \ref{corBM} combined with applications of the multidimensional functional central limit theorem, carried out in Lemma \ref{TMcon}. To extend convergence in distribution to the convergence of moments in Theorem \ref{momeN}, we prove in Proposition \ref{uin} uniform integrability of the sequences of random variables in (\ref{eqA}) and (\ref{eqB}). The proof of Proposition \ref{uin} is  carried out in two steps and requires some preparatory work, which we take up in Subsection \ref{prep}. 

The first step comprises of proving this statement in dimension  $N=1$, and is done in Lemma \ref{lmomEK}.
We note that, regarding 1-dimensional, arbitrary zero-mean increment-distribution random walks, [Theorem \ref{momeN}, (\ref{eqA})] is already shown in [$\mathsection$ 23, Spitzer \cite{Sp76}],  combine Propositions 3 and 6 there. Our proof of uniform integrability of the sequence in (\ref{eqA}) in dimension $N=1$ is based upon this result, and is done in [(\ref{mom1}), Lemma \ref{lmomEK}]. Our proof of uniform integrability of the sequence in (\ref{eqB}) in dimension $N=1$ uses a maximal inequality, and is done in [(\ref{aimm}), Lemma \ref{lmomEK}].  The second step for deriving this Proposition comprises of the couplings in Lemma \ref{propcoup}. The reason we require this is that, unlike the continuous-time $N$-dimensional simple random walk, the coordinates of which are independent 1-dimensional simple random walks, see, for instance, Proposition 1.2.2 in \cite{LL10}, the coordinates of the discrete-time $N$-dimensional simple random walk are clearly dependent. These couplings allow to extend the uniform integrability results in Lemma \ref{lmomEK} to $N$ dimensions. These preparatory Lemmas  \ref{propcoup} and \ref{lmomEK} comprise Subsection \ref{prep}.

As noted already, the interconnection among [(\ref{Ftil}), Theorem \ref{momeN}] and [(\ref{Gtil}), Theorem \ref{momeN}] is due to that exit-times are times that new partial-maxima values are achieved for the random walk. The fact that these two limit theorems may be associated by coupling is already pointed out and exploited in the context of 1-dimensional random walks in [Theorem 3, $\mathsection$ 23, \cite{Sp76}]. We note that our proof approach in Proposition \ref{corBM} is facilitated by exploiting the corresponding coupling connection directly for the Brownian motion limiting objects.  Further, we note that the method of deriving Proposition \ref{propFG}, we mention below next, is also an instance of this approach in dimension one.

In Section \ref{EKsec}, we revisit Theorem \ref{EK}. In Subsection \ref{subEK}, we observe that two different routes for proving [(\ref{PHI}),  Theorem \ref{EK}] and [(\ref{GAM}),  Theorem \ref{EK}] are possible, due to their interconnection which we point out to in Proposition \ref{propFG} there. 
An easy consequence of Theorem \ref{EK} combined with Proposition \ref{propFG} is also derived there in Corollary \ref{propmax}.  This corollary is used later in the proof of Proposition \ref{corBM}. By means of Proposition \ref{propFG}, various different routes for deriving Theorem \ref{EK} are made available. We collect them together, along with various associated pointers to the literature, in Subsection \ref{routeM}.

In Subsection \ref{app}, we  show an elementary proof approach to Theorem \ref{EK} regarding zero-mean increment-distribution random walks. The first statement we give there is two-fold and provides, in [(\ref{sect}), Theorem \ref{tauT}], with the convergence of Laplace transforms corresponding to [(\ref{TAe}), Theorem \ref{EK}], and, in [(\ref{sigL}), Theorem \ref{tauT}] with the convergence of Laplace transforms corresponding to the so-called first-passage times.  We derive a proof of the Theorem \ref{EK}, from [(\ref{sect}), Theorem \ref{tauT}] in conjunction with Proposition \ref{propFG} by invoking some known facts from [10, p.\ 273, \cite{Sp76}].  Further, we derive in Corollary \ref{stable} the limit law of first-passage times as another direct byproduct. In this way, our approach brings together Theorem \ref{EK} with the celebrated stable law exponent 1/2 for first passage times. The remainder of Subsection \ref{app} is then devoted to devising an elementary proof of (both parts of) Theorem \ref{tauT}, which to our knowledge is not developed elsewhere (cf.\ Remark \ref{standLBM}), in order for our approach there to be elementary in its entirety.  This proof relies on the invariance principle and comprises of deriving the associated Laplace transforms in the simple random walk case, by building upon variants of known arguments relying on Doob\apost s Optional Stopping Theorem, cf.\ [$\mathsection$ 10.12, Williams \cite{W91}], along with a simplifying detour via Lemma \ref{ll2}, which in effect follows known arguments, cf.\ for instance, [Corollary 2.17, Kallenberg \cite{Kll97}]. By Lemma \ref{limlam}, which extends basic calculus results suggested in [$\mathsection$ 1.3, Lalley \cite{Ll}], the asymptotics of both these Laplace transforms are then derived there.

\section{The Erd\H os and Kac theorem revisited}\label{EKsec} 

\subsection{A closely related observation and a consequence of it}\label{subEK}

Let $\Phi(\cdot)$ and $\Gamma(\cdot)$ be as in Theorem \ref{EK}; these two distributions are associated as follows.

\begin{proposition}\label{propFG} \mbox{ } $\Phi(t) = 1- \Gamma\left( \frac{1}{ \sqrt{t}}\right)$, $t > 0$. 
\end{proposition}
\begin{proof} Let $M(t) = \sup_{s \in [0,t]}|W(s)|$ and observe that the following equality holds
\begin{equation}\label{coupTM}
\{T < t\} = \{M(t) > 1\}.
\end{equation}
However, from the Brownian scaling property $W(s) \stackrel{d}{=} c  W(s/c^{2})$, for all $c>0$, we have the following. 
\begin{lemma}\label{MW} $M(t) \stackrel{d}{=} \sqrt{t} M(1).$
\end{lemma}
\begin{proof}[Proof of Lemma \ref{MW}]
By Brownian scaling, we have that
\begin{eqnarray*}\label{eq:}
M(t) & \stackrel{d}{=} &  c  \sup_{s \in [0,t]}|W(s/c^{2})|  \\
& = & c M(t/c^{2}), 
\end{eqnarray*}
therefore, plugging $c=\sqrt{t}$ in the display above proves the statement. 
\end{proof}
\noindent Note that $(\ref{coupTM})$ together with Lemma \ref{MW} give
\begin{eqnarray}\label{sc}
\Phi(t) := \pr(T < t ) & = & \pr(M(t) > 1) \nonumber \\
& = & \pr\left( M(1) > \frac{1}{ \sqrt{t}}\right) \nonumber \\
& = & 1- \Gamma\left( \frac{1}{ \sqrt{t}}\right), 
\end{eqnarray}
where in (\ref{sc}) we used continuity of $\Gamma\left( \cdot \right)$. The proof of Proposition\ref{propFG} is thus complete. 
\end{proof}

For ease of reference below, we record the following consequence of Theorem \ref{EK} by Proposition \ref{propFG}.  

\begin{corollary}\label{propmax} Let $M(t) = \sup_{s \in [0,t]}|W(s)|$ and let $\Gamma_{t}(x) = \pr(M(t)<x)$. We have that 
\[
\Gamma_{t}(x) = \Gamma\left(\frac{x}{\sqrt{t}}\right) = H\left(\frac{t}{x^{2}}\right),
\]
where $H$ is as in (\ref{eq00}), or (\ref{eq01}). 
\end{corollary} 

\begin{proof}
This follows from Lemma \ref{MW} above combined with [(\ref{GAM}), Theorem \ref{EK}].  
\end{proof}

\subsection{A miscellany of proof approaches to Theorem \ref{EK}}\label{routeM}
In this section we comment and give pointers to the literature regarding various possible proof approaches to Theorem \ref{EK}. Limit theorems [(\ref{TAe}), Theorem \ref{EK}] and [(\ref{Be}), Theorem \ref{EK}] follow by applications of Donsker\apost s Theorem (cf.\ Lemma \ref{TMcon}, in the proof of Theorem \ref{momeN} below, which shows their higher-dimensional analogues).  Hence, we focus here on the remaining parts of the statement, [(\ref{PHI}), Theorem \ref{EK}] and [(\ref{GAM}), Theorem \ref{EK}]. Remarks \ref{SPR}, \ref{FKR}, and \ref{levrem} regard proof approaches to (\ref{PHI}) and to (\ref{GAM}), which we note that are interconnected via Proposition \ref{propFG}. In Remark \ref{EQR}, we comment on the equivalence of $H(\cdot)$ as in (\ref{eq01}) and as in (\ref{eq00}). 

\begin{remark}\label{SPR}\textup{Regarding deriving (\ref{PHI}) for $H(\cdot)$ as in (\ref{eq00}), one can follow [$\mathsection$ 21, Proposition 5, \cite{Sp76}], which we commented upon in the last paragraph of $\mathsection$ \ref{sub12}, along with (\ref{TAe}).}
\end{remark}

\begin{remark}\label{FKR}\textup{An approach leading to (\ref{GAM}) for $H(\cdot)$ as in (\ref{eq00}) is through the Feynman-Kac formulas for Brownian motion and showing a solution of the heat equation by use of the separation of variables technique; cf.\ the argument following [Theorem 7.45, M\"orters and Peres \cite{MP10}].}
\end{remark}

\begin{remark}\label{levrem}
\textup{Various proofs of L\'evy\apost s triple law, which we mentioned in the second paragraph of $\mathsection$ \ref{sub12} and which lead to (\ref{GAM}) for $H(\cdot)$ as in (\ref{eq01}), are follows. One approach goes through first deriving the associated law for the simple random walk (either by induction, or the technique of repeated reflections, aka method of images), and then using the classic central limit theorem and Donsker\apost s Theorem; cf.\ [Chpt. 2, $\mathsection$ 11.1, Billingsley \cite{Bl68}]. A closely related proof approach goes through arguments employing the said technique directly for standard Brownian motion instead, see for instance, [Theorem 8.7.3, Durrett \cite{Dr10}], or [$\mathsection$ 6.5, Schilling and Partzsch \cite{SR14}]. A different approach goes through the Feynman-Kac formulas for stochastic integrals and a method for verifying a solution to the heat equation, which may be motivated by heuristics, cf.\ [Theorem 7.45, M\"orters and Peres \cite{MP10}].}
\end{remark}

\begin{remark}\label{EQR}
\textup{For showing the equivalence of $H(\cdot)$ as in (\ref{eq01}) with that in (\ref{eq00}), one can combine the argument in  Remark \ref{FKR} with L\'evy\apost s triple law, for the various routes to which we refer to in Remark \ref{levrem}. Alternatively, through Proposition \ref{propFG}, one can employ the statement we infer Remark \ref{SPR}, along with L\'evy\apost s triple law. It is also possible to transform $H(\cdot)$ as in (\ref{eq01}) to $H(\cdot)$ as in (\ref{eq00}), see, for instance, [Lemma 11.6, Schilling and Partzsch \cite{SR14}], which relies on a Fourier expansion, or the references pointed out in the footnote in [p.\ 80, Billingsley \cite{Bl68}]. A different approach, which is connected to the proof of Theorem \ref{EK} below, is pointed out in [10, p. 273, \cite{Sp76}].}
\end{remark}
\subsection{An elementary proof approach to Theorem \ref{EK}}\label{app}

\begin{theorem}\label{tauT}
Let $\tau_{b}$, $T$, and $S_{t}$ be as in Theorem \ref{EK}. Let also $S = \inf\{ t: W(t) = 1\}$ and $\sigma_{b} = \inf\{t: S_{t} = b\}$, $b\geq1$. We have that
\begin{equation}\label{sect}
\lim_{b \rightarrow \infty} \E\big(e^{-\theta \frac{\tau_{b}}{b^{2}}}\big) = \E(e^{-\theta T})  = \sech{\sqrt{ 2 \theta}},  
\end{equation}
where $\sech\theta = (\cosh\theta)^{-1} = \frac{2}{e^{\theta} + e^{-\theta}}$, and
\begin{equation}\label{sigL} 
\lim_{b \rightarrow \infty} \E\big(e^{- \frac{\theta \sigma_{b}}{b^{2}}}\big) = \E(e^{-\theta S}) = \expo(-\sqrt{2 \theta}), 
\end{equation}
$\theta >0$.
\end{theorem}

\begin{proof}[Proof of Theorem \ref{EK}.]
By Proposition \ref{propFG}, it suffices to show (\ref{PHI}) for $H(\cdot)$ as in (\ref{eq00}) and as in (\ref{eq01}). We may derive (\ref{PHI}) from (\ref{sect}) by invoking the following expansion in series of simple functions 
\begin{equation}\label{arrp1}
\sech {\sqrt{ 2 \theta}} = \frac{\pi}{2} \sum_{n = 0}^{\infty} (-1)^{n} \frac{2n+1}{\theta + \frac{\pi^{2}}{8}(2n+1)^{2}}, 
\end{equation}
or, alternatively, the geometric series representation as follows
\begin{equation}\label{sec2}
\sech\sqrt{ 2 \theta} = \frac{2 e^{-\sqrt{ 2 \theta}}}{1-(-1)e^{-2\sqrt{ 2 \theta}}} = 2 \sum_{k \geq 0} (-1)^{k} e^{-(2k+1)\sqrt{ 2 \theta}}.
\end{equation}
\noindent  The expansion in (\ref{arrp1}) leads to (\ref{PHI}) with $H(\cdot)$ in the form (\ref{eq00}), whereas (\ref{sec2}) leads (by term-to-term inversion) to it with $H(\cdot)$ in the form (\ref{eq01}); cf.\ [\cite{Sp76}, p. 273, 10]. 
\end{proof}

We state an immediate byproduct of Theorem \ref{tauT}.  Let $f_{S}(\cdot)$ denote the probability density associated to $S$.  
\begin{corollary}\label{stable} 
$\dis{ f_{S}(t) = \frac{1}{\sqrt{2\pi t^{3}}} \exp(-1/2t), t\geq 0.}$
\end{corollary}
\begin{proof}[Proof of Corollary \ref{stable}.]
We appeal to the known fact that the Laplace transform in (\ref{sigL}) may be inverted; cf., for instance, [$\mathsection$ 9,  Chpt. 2, \cite{RW93}].
\end{proof}

\begin{proof}[Proof of Theorem \ref{tauT}.] The proofs of the left-hand-sides parts of (\ref{sect}) and (\ref{sigL}) are omitted 
since for positive random variables convergence of Laplace transforms is equivalent to convergence in distribution, and thus the former is equivalent to (\ref{TAe}), whereas the latter follows by an application of Donsker\apost s Theorem (see, for instance, [Example 8.6.6, \cite{Dr10}]). From this theorem we have that it thus suffices to show the remaining parts of the claim for the simple random walk, defined as follows. Let $S_{t} = \sum_{s=1}^{t} \zeta_{s}$, where $(\zeta_{i}: i \geq 1)$ uniformly distributed $\{-1,1\}$-valued independent random variables. 
\begin{proposition}\label{T1}
Let $\dis{\lambda(z) = \frac{1 - \sqrt{1-z^{2}}}{z}}$, $z \in (0,1)$. We have that
\begin{equation}\label{eqG3}
\E(z^{\tau_{b}}) = \frac{2}{\lambda^{b}(z) + \lambda^{-b}(z)}, 
\end{equation}
and that 
\begin{equation}\label{eqzT} 
\E(z^{\sigma_{b}}) = \lambda^{b}(z).
\end{equation}
\end{proposition}

\begin{proof}[Proof of Proposition \ref{T1}.] We require the two known statements following next. Their proofs are short and elegant and are given for completeness after the proof of Theorem \ref{tauT}. Note that Lemma \ref{ll2} regards recurrence of $(S_{t})$. 
  
\begin{lemma}\label{ll1}
$S_{\tau_{b}}$ and $\tau_{b}$ are mutually independent.
\end{lemma}

\begin{lemma}\label{ll2} $\dis{\tau_{b}<\infty}$ and $\dis{\sigma_{b}<\infty}$, for all $b$, a.s.. 
\end{lemma} 

\noindent Let $M_{n}^{\theta} = (\sech\theta)^{n} e^{\theta S_{n}}$, $\theta>0$. 
We have that $M_{n}^{\theta}$ is a product of independent, mean 1 random variables, and hence a martingale. By Lemma \ref{ll1} and Doob\apost s Optional Stopping Theorem (cf., for instance, [$\mathsection$ 10.10, Theorem (b), (ii), \cite{W91}]), whose hypotheses are satisfied due to that $|M_{n \wedge \tau_{b}}| \leq e^{\theta b}$, since $\sech\theta < 1$, and $\tau_{b}<\infty$ a.s., we have that 
\begin{equation}\label{eqta1}
\E(\sech \theta)^{\tau_{b}} = \sech(\theta b).
\end{equation}
%\frac{2}{e^{\theta b} +  e^{-\theta b}}
%= \frac{2}{\frac{1}{\psi}+ \frac{\sqrt{1-\psi^{2}}}{\psi}^{b} + \frac{1}{\psi}+ \frac{\sqrt{1-\psi^{2}}}{\psi}^{-b}} 
By Doob\apost s Optional Stopping Theorem, whose hypotheses are again satisfied due to that $|M_{n \wedge \sigma_{b}}| \leq e^{\theta b}$, and, by Lemma \ref{ll2}, $\sigma_{b}<\infty$ a.s., we have that
\begin{equation}\label{eqs}
\E(\sech \theta)^{\sigma_{b}} = \exp(-\theta b).
\end{equation}
Setting $\sech \theta = z$ and since $z \in (0,1)$ and $\E(z)^{\sigma_{b}}, \E(z)^{\tau_{b}} \in [0,1]$, yields $\theta = \ln(1/\lambda(z))$; substituting this in (\ref{eqta1}) and in (\ref{eqs}) gives (\ref{eqG3}) and (\ref{eqzT}) respectively. Thus, the proof is complete. 
\end{proof} 

\begin{lemma}\label{limlam} Let $\lambda(z)$ be as in Proposition \ref{T1}. We have that
\begin{equation*}
\lim_{b \rightarrow \infty}[\lambda(e^{- \theta / b^{2}})]^{b} = e^{- \sqrt{2 \theta}}, \mbox{ } \theta>0. 
\end{equation*}
\end{lemma}
\begin{proof}[Proof of Lemma \ref{limlam}.] Observe that\footnote{as usual, we write $f(z) \sim g(z)$ as $z\rightarrow z_{o}$ to denote $\dis{\lim_{z \rightarrow z_{o}} \frac{f(z)}{g(z)} = 1}$.} $\dis{1-\lambda(z) \sim \sqrt{2}{\sqrt{1-z}}}$, as $z \rightarrow 1^{-}$, and hence
\begin{equation}\label{eqexp}
b \left(1 - \lambda(e^{- \frac{\theta}{b^{2}}})\right) \sim b\sqrt{2}{\sqrt{1-e^{-\frac{\theta}{b^{2}}}}}, \mbox{ as } b \rightarrow \infty,
\end{equation}%\textbf{$z = e^\frac{- \theta}{b^{2}}$}
However, we have that 
\begin{equation}\label{eqlast}
\lim_{b \rightarrow \infty} b \sqrt{2} \sqrt{1-e^{- \frac{\theta}{b^{2}}}} = \sqrt{2 \theta},
\end{equation}
since,  by the Maclaurin series expansion, $\dis{1- e^{-x} = \sum_{k \geq 1}\frac{(-1)^{k-1} x^{k}}{k!}}$. Combining $(\ref{eqexp})$ and $(\ref{eqlast})$ gives
\[
\lim_{b \rightarrow \infty} \left(1- \left(1 -\lambda(e^{- \frac{\theta}{b^{2}}})\right)\right)^{b} = e^{-\sqrt{2  \theta}},
\]
due to that $\dis{\lim_{n \rightarrow \infty} \left( 1+\frac{a(n)}{n} \right)^{n}= e^{w}}$, $\dis{w = \lim_{n \rightarrow \infty} a(n)}$. The last display completes this proof. 
\end{proof}

\noindent To finish the proof, simply note that Lemma \ref{limlam} yields the right-hand-side equality in (\ref{sigL}) from (\ref{eqzT}), and 
that in (\ref{sect}) from (\ref{eqG3}) and an application of the algebraic limit theorem.  
\end{proof}

\begin{proof}[Proof of Lemma \ref{ll1}]
Note that 
\[
\pr(\tau_{b} = t, S_{\tau_{b}} = \pm b) = \pr(\tau_{b} = t) - \pr(\tau_{b} = t, S_{\tau_{b}} = \mp b),
\]
and hence, by symmetry, $\dis{\pr(\tau_{b} = t, S_{\tau_{b}} = \pm b) = \frac{1}{2} \pr(\tau_{b}= t)}$, as required. 
\end{proof}

\begin{proof}[Proof of Lemma \ref{ll2}] 
By Kolmogorov\apost s zero-one law (cf., for instance, [$\mathsection$ 4.11, Theorem, (ii), \cite{W91}]) we have that $\limsup_{n \rightarrow \infty} S_{n} = c$, a.s., where $c$ is a constant such that $c \in [-\infty, \infty]$. Clearly, it suffices to show that $c = +\infty$. Since also $\limsup_{n \rightarrow \infty} S_{n+1} = c$, a.s., we have that if $c<\infty$, then $\zeta_{1} = 0$ a.s., and hence, by contradiction, $|c| = \infty$.  Finally, if $c= - \infty$ a.s., then, by symmetry, $S_{n}' := -S_{n} \stackrel{d}{=} S_{n}$, and hence $\liminf_{n \rightarrow \infty}(S_{n}) = - \limsup_{n \rightarrow \infty}(-S_{n}) = +\infty$, which leads to the contradiction, $+\infty \leq -\infty$, and the proof is complete.
\end{proof}

\begin{remark}\label{standLBM}\textup{Other proofs of Theorem \ref{tauT} rely on Donsker\apost s theorem and typically deal directly with Brownian motion to derive first the right-hand-side equalities in (\ref{sect}) and in (\ref{sigL}). A proof of Theorem \ref{tauT} may be thus shown by relying on applications of Doob\apost s Optional Stopping Theorem for so-called exponential martingales, see for instance, [Proposition 3.7, Chpt.\ II, Revuz and Yor \cite{RY94}], or [Theorem 8.5.7, Durrett \cite{Dr10}]. Another, heuristic proof approach to (\ref{sigL}) can be found in [$\mathsection$ 13.7, Breiman \cite{Br92}].}
\end{remark}

\section{Proofs of main results}\label{31p}

\subsection{Proof of Proposition \ref{corBM}}

\begin{proof}[Proof of Proposition \ref{corBM}]
We let $({\bf W}_{t}: t \in \R_{+})$ be standard Brownian motion on $\R^{N}$ with covariance matrix $I$. We let also $M_{t} = \sup_{0 \leq s \leq t}|{\bf W}_{s}|$. From, for instance, Lemma 3.4.1 in \cite{LL10}, we have that, if we let $(W_{n, t}; n = 1, \dots, N)$ be standard mutually independent Brownian motions in $\R$, and if $v_{i} \in \R^{N}$, $i = 1, \dots, N$, then
\begin{equation}\label{lin}
{\bf W}'_{t} := \sum_{n=1}^{N} v_{n}W_{n,t}, 
\end{equation}
is Brownian motion in $\R^{N}$ with covariance matrix $\Sigma = VV^{T}$ and $V= [v_{1}, \dots, v_{N}]$. A standard consequence of (\ref{lin}) is that ${\bf W}_{t}$ is indeed of the form ${\bf W}_{t} = (W_{1,t}, \dots, W_{N,t})$.  Hence, letting 
$M_{t} = \sup_{0 \leq s \leq t}|{\bf W}_{s}|$ and $m_{n,t} = \sup_{s \in [0,t]}|W_{n,t}|$, we have that $(m_{n,t}: n = 1, \dots, N)$ are independent and identically distributed according to $\Gamma_{t}$ in Corollary \ref{propmax}. Note that, this gives that
\begin{eqnarray}\label{finbm}
\pr(M_{t} < x) &=& \prod_{n=1}^{N} \pr(m_{n,t} < x) \nonumber \\
& =& \left(H\left(\frac{t}{x^{2}}\right)\right)^{N}.
\end{eqnarray}
We prove (\ref{Gtil}). Recall that $\wdh{\bf{W}}_{t}$ 
is $N$-dimensional Brownian motion with covariance matrix $\frac{1}{N}I$. Hence, another application of (\ref{lin}) gives that
\begin{equation}\label{eqWD}
\wdh{\bf{W}}_{t} \stackrel{d}{=} \frac{1}{N} {\bf W}_{t},
\end{equation}
and in particular, that $\wdh{M} \stackrel{d}{=} \frac{1}{N} M_{1}$, which from (\ref{finbm}) gives  (\ref{Gtil}). 

We prove (\ref{Ftil}). We let $\wdh{M}_{t} = \sup_{0 \leq s \leq t}|{\bf \wdh{W}}_{s}|$. Note that  
\begin{eqnarray}\label{eqmt}
\pr(\wdh{T}<t)&=&\pr(\wdh{M}_{t} >1) \nonumber \\
& = & \pr\left(\frac{1}{N} M_{t} >1\right) \nonumber \\
& = & 1- \pr\left( \frac{1}{N} M_{t} <1\right),
\end{eqnarray}
where the first equality follows by considering the corresponding equality of events, the second one comes from applying (\ref{eqWD}), and in the third one we use continuity of $M_{t}$. Therefore, from (\ref{finbm}) and (\ref{eqmt}), we derive (\ref{Ftil}), which completes the proof. 
\end{proof}

\subsection{Proof of Theorem \ref{momeN}}

Prior to giving the proof of Theorem \ref{momeN} in paragraph \ref{subthmprf} below, we first focus on deriving the result stated next which is required and plays a key role in that proof. Recall first the definitions of $\til{T}_{N, r}$ and $\til{M}_{N, t}$ from (\ref{defTtil}) and (\ref{defMtil}) respectively. 
We note that, henceforth, since the dimension $N \geq 1$ will be a fixed, finite integer which will be implicit from context, we simply write $\til{T}_{r}$ and $\til{M}_{t}$ respectively instead, dropping the subscript associated with the dimension, in order to simplify notation.

\begin{proposition}\label{uin} 
We have that
\begin{equation}\label{intA}
\left\{ \left( \frac{\til{T}_{r}^{p}}{r^{2p}},  r \geq 1\right) \right\}  \mbox{ is uniformly integrable for all } p\geq 1,
\end{equation}
and further, that 
\begin{equation}\label{intB}
\left\{  \left( \frac{\til{M}_{t}^{p}}{t^{p/2}}, t \geq 1\right)  \right\}  \mbox{ is uniformly integrable for all } p\geq 1.
\end{equation}
\end{proposition}

The proof of Proposition \ref{uin} is postponed to  \ref{subsecunifint}, whereas in the subsequent  \ref{prep} we state and prove two Lemmas that we require in the proof of Proposition \ref{uin}. 
 
\subsubsection{Two preparatory Lemmas} \label{prep}
In this section we state first and prove below next two preparatory results that we invoke later, in the proof of Proposition \ref{uin}. Lemma \ref{propcoup} given next regards a coupling connection among the simple $N$-dimensional random walk and a collection of $N$ independent one-dimensional simple random walks.  The second one regards uniform integrability of the sequences considered in Proposition \ref{uin} in one dimension only. To state Lemma \ref{propcoup} next recall that we let $(\Ze_{t}: t \geq 0)$  be the simple $N$-dimensional random walk.  The proof of this Lemma given here invokes a simple realization associated to its description in (\ref{eq2}), see for instance, [$\mathsection$ 1.2.4, \cite{St05}].

\begin{lemma}\label{propcoup}
Let $\{(Z_{n, t}, t \geq 1): n = 1, \dots, N\}$ be a collection of independent simple one-dimensional random walks. Let $\tau_{n,r} = \inf\{t: |Z_{n, t}| > r\}$, where $n=1, \dots, N$. Let also $m_{n,t} = \max_{1 \leq s \leq t}|Z_{n,s}|$. We have that we may define $(\Ze_{t}: t \geq 0)$ on the same probability space with $\{(Z_{n, t}, t \geq 1): n = 1, \dots, N\}$, such that 
\begin{equation}\label{coupl1}
\til{T}_{r} \leq \sum_{n=1}^{N} \tau_{n,r}- (N-1);
\end{equation}
and, such that $\til{M}_{t} \leq \max_{n=1, \dots, N} m_{n,t}$, and, a fortiori,  
\begin{equation}\label{coupl2}
\til{M}_{t} \leq \sum_{n=1}^{N} m_{n,t},   
\end{equation}
almost surely. 
\end{lemma}

\begin{lemma}\label{lmomEK} Let $S_{t} = \sum_{s=1}^{t} \zeta_{s}$ be the simple one-dimensional random walk. Furthermore, we let $\tau_{b} = \inf\{ t: |S_{t}| \geq b\}$ and $m_{t} = \max\{|S_{1}|, \dots , |S_{t}|\}$. We have that
\begin{equation}\label{mom1}
\left\{ \left(\frac{\tau_{b}}{b^{2}}\right)^{p}, b \geq 1\right\} \mbox{ is uniformly integrable for all } p\in \Na,
\end{equation}
and that 
\begin{equation}\label{aimm}
\left\{\left(\frac{m_{t}^{r}}{t^{r/2}}\right), t \geq 1\right\} \mbox{ is uniformly integrable for all } r \in \Na,
\end{equation}
\end{lemma} 

\begin{proof}[Proof of Lemma \ref{propcoup}.]
We obtain a realization of $(\Ze_{t}: t \geq 1)$ as follows. The value of $\Ze_{t+1}- \Ze_{t}$ is decided by first choosing one of the $N$ coordinates uniformly at random, and then deciding whether it is to be $\{+1,-1\}$. With this in mind, let $(d_{s}, s \geq 1)$ be a $\{1, \dots, N\}-$valued independent and uniformly distributed collection of random variables, which is independent of $\{(Z_{n, t}, t \geq 1): n = 1, \dots, N\}$. Letting $K_{n,t} = \sum_{s = 1}^{t} I(d_{s} = n)$,  $n = 1, \dots, N$, we have that  
\begin{equation}\label{repr}
\Ze_{t} =  (Z_{1, K_{1,t}}, \dots , Z_{N, K_{N,t}}).
\end{equation}

\noindent Let $\delta = \{i: |\Ze_{\til{T}_{r}}(i)|>r\}$, $\Ze_{t}(i) := Z_{i,t}$. From (\ref{repr}) we have that 
\begin{equation}\label{Trt}
\til{T}_{r} = \tau_{i,r}+ \sum_{j: j \not= i } K_{j, \til{T}_{r}} \mbox{ on } \{\delta = i\}  
\end{equation}
for all $i \not= j$; however, we have by the definition of $\delta$, that
\begin{equation}\label{eqKj}
 K_{j, \til{T}_{r}} \leq \tau_{j, r}-1 \mbox{ on } \{\delta = i\}, 
\end{equation}
for all $i \not= j$. Therefore, (\ref{Trt}) and (\ref{eqKj}) imply (\ref{coupl1}).

Let $(\Se_{t}: t \geq 1)$ be the concatenation $\Se_{t} = (Z_{1, t}, \dots, Z_{N, t})$. Let $M_{t}= \max_{1 \leq s \leq t}|\Se_{s}|$ and let also $\til{m}_{n,t} = \max_{1 \leq s \leq t}|Z_{n,K_{n,s}}| $. From (\ref{repr}) because $K_{n,t} \leq t$, for all $n$ and $t$, we have that, for every $n$,
\[
\til{m}_{n,t} \leq  m_{n,t} \mbox{ for all } t, 
\]
almost surely, and hence 
\[
\til{M}_{t}:= \max_{n=1, \dots, N} \til{m}_{n,t} \leq \max_{n=1, \dots, N} m_{n,t} =: M_{t}. 
\]
which proves the statement preceding  (\ref{coupl2}). Hence the proof is complete. 
\end{proof}

\begin{proof}[Proof of Lemma \ref{lmomEK}.] For both parts we will verify the assumptions of the following general result, which corresponds to the reverse part of [Theorem 5.4, Chpt.\ 1, \cite{Bl68}]. 

\begin{lemma}\label{Billem}
Let $(X_{n}; n \geq0)$ be a sequence of non-negative and integrable random variables. Suppose that $X_{n}  \xrightarrow{d} X$ and that $\E(X_{n}) \rightarrow \E(X)$, as $n \rightarrow \infty$, as well as that $\E(X)<\infty$.  Then $(X_{n}; n \geq0)$ are uniformly integrable. 
\end{lemma}
\noindent From Theorem \ref{EK} the continuous mapping theorem yields 
\begin{equation}\label{Td}
\left(\frac{\tau_{b}}{b^{2}}\right)^{p} \xrightarrow{d} T^{p} \hs \mbox{ as } b \rightarrow \infty,
\end{equation}
where the distribution of $T^{p}$ is such that $\dis{\Phi_{p}(x) = \pr(T^{p} \leq x) = 1-  H(x^{1/p})}$,  and $H$ is as in (\ref{eq00}). 
From a simple computation from the density function of $T^{p}$, we have that 
\begin{equation*}\label{momEK3}
\E(T^{p}) = p! \cdot \frac{\pi}{2} \left(\frac{8}{\pi^{2}}\right)^{(p+1)} \sum_{k=0}^{\infty} (-1)^{k}\left(\frac{1}{2k+1}\right)^{2p+1},
\end{equation*} 
cf.\  [Proposition 3, $\mathsection$ 23, \cite{Sp76}], and hence 
\begin{equation}\label{Cbil1}
\E(T^{p})<\infty. 
\end{equation}
Furthermore,  if $L(\theta)= \E(e^{-\theta \frac{\tau_{b}}{b^{2}}})$, letting $L^{(p)}(\theta)$ denoting its derivative of the $p$-th order, we have 
\[
\E\left(\frac{\tau_{b}}{b^{2}}\right)^{p} = (-1)^{p} \cdot \lim_{\theta \rightarrow 0} L^{(p)}(\theta)  
\]
and thus, from [(\ref{eqG3}), Proposition \ref{T1}] we get that the integrability condition of Lemma \ref{Billem} is granted, which is that
\begin{equation}\label{Cbil2}
\E\left(\frac{\tau_{b}}{b^{2}}\right)^{p}<\infty, 
\end{equation}
for all $b \geq 1$. In addition, invoking [Propositions 3 and 6, $\mathsection$ 23, \cite{Sp76}] we have that 
\begin{equation}\label{Te}
\E\left(\frac{\tau_{b}}{b^{2}}\right)^{p}  \rightarrow \E(T^{p}), \mbox{ as } b \rightarrow \infty. 
\end{equation}
Thus, from (\ref{Td}), (\ref{Cbil1}), (\ref{Cbil2}), and (\ref{Te}) the assumptions of Lemma \ref{Billem} are satisfied, hence yielding (\ref{mom1}).  

For the second part, invoking the Lemma in [p.\ 69, \cite{Bl68}] gives
\begin{equation}\label{bilmax}
\pr\left(\frac{m_{t}}{\sqrt{t}} \geq x\right) \leq 2 \hspace{0.5mm}\pr\left(\frac{2 |S_{t}|}{\sqrt{t}} \geq x\right), 
\end{equation}
for all $x > 2 \sqrt{2}$. Letting $\dis{\bar{\Theta}_{t, r}(x) = \pr\left(\frac{m_{t}^{r}}{t^{r/2}} \geq x\right)}$ and $\dis{\bar{\Delta}_{t, r}(x) = \pr\left(\frac{2 |S_{t}|^{r}}{t^{r/2}} \geq x\right)}$, (\ref{bilmax}) yields that
\begin{equation}\label{GamDel}
\bar{\Theta}_{t, r}(x) \leq 2 \bar{\Delta}_{t, r}(x), \mbox{ for all } x > (2 \sqrt{2})^{r},
\end{equation}
$r >0$. Further, we will show that
\begin{equation}\label{equis}
\dis{\left\{ \left(\frac{2|S_{t}|}{\sqrt{t}}\right)^{p}, t \geq 1\right\} \mbox{ is uniformly integrable for all } p \in \Na,}
\end{equation}
and hence, from the definition of uniform integrability, (\ref{equis}) and (\ref{GamDel}) yield (\ref{aimm}). 

To complete the proof, it remains to show (\ref{equis}), for which we will again check the hypotheses of Lemma \ref{Billem}. We claim that
\begin{equation}\label{momclt}
\E \left(\frac{|S_{t}|}{\sqrt{t}}\right)^{p} \rightarrow \E|N|^p, \mbox{ as } t \rightarrow \infty,
\end{equation}
for all $p \in \Na$, where $N$ denotes a standard normal. To show (\ref{momclt}) note that, for all $p \geq 2$, this is a direct consequence of the second part of the statement in [Theorem 7.5.1, \cite{G12}], since we have that $\E(|\zeta_{i}|^{p})=1$. Whereas the case $p=1$ follows from the central limit theorem, cf. Theorem 7.1.1 in \cite{G12} and checking the general conditions for convergence in distribution to extend to convergence of moments, cf.\ for instance Theorem 5.5.1 in \cite{G12}. To do this simply note that $\E(|\zeta_{i}|)=1$, and hence that $(|\zeta_{i}|: i \geq 1)$ are uniformly integrable since these are uniformly bounded by an integrable random variable, cf.\ for instance, Theorem 5.4.4 in \cite{G12} . Furthermore, from the central limit theorem and the continuous mapping theorem, cf. Theorem 5.10.4 in \cite{G12}, we have that the convergence in distribution analogue of (\ref{momclt}) holds. In addition, we have that $\E\left(\frac{|S_{t}|}{\sqrt{t}}\right)^{p}< \infty$ from a consequence of the Marcinkiewicz--Zygmund
inequalities, Corollary 3.8.2 in \cite{G12}, and that $\E|N|^{p} < \infty$. From these facts we have that (\ref{equis}) follows, and the proof is complete. 
\end{proof} 

\begin{remark}
\textup{We note that an alternative route to show (\ref{equis}) may be concocted by consulting the proof of [Theorem 7.5.1, \cite{G12}]; from the first display there, we have that (\ref{equis}) holds for all $p \geq 2$ since $\E(|\zeta_{i}|^{p})=1$, whereas (\ref{equis}) for $p=1$ follows again from this fact by a simple criterion  for uniform integrability, cf.\ for instance [Theorem 5.4.2, \cite{G12}].}
\end{remark}

\subsubsection{Uniform integrability in higher dimensions}\label{subsecunifint}

\begin{proof}[Proof of Proposition \ref{uin}.]
We will need the following generic Lemma. 
\begin{lemma}\label{gener}
If $\{\left(X_{n,k}^{p}\right)^{\infty}_{k=1}; n = 1, \dots, N\}$  is a collection of uniformly integrable sequences of positive random variables for every $p \in \Na$, then $\left(\left(\sum_{n=1}^{N} X_{n,k}\right)^{p}\right)_{k=1}^{\infty}$ is uniformly integrable for every $p \in \Na$.       
\end{lemma}
\begin{proof}[Proof of Lemma \ref{gener}.]
The proof proceeds by induction in $N$. Clearly the statement holds for $N=1$. We assume that it holds for $N-1$.  Let $Y_{k}: = \sum_{n=1}^{N-1} X_{n,k}$ and $W_{k}:=X_{N,k}$. Since finite sums of uniformly integrable random variables are uniformly integrable (cf., for instance, Theorem 5.4.6 in \cite{G12}), an application of binomial theorem gives that it suffices to show that $c Y_{k}^{a}W_{k}^{b}$ are uniformly integrable for all $a$ and $b$ non-negative integers such that $a+b = p$, and where $c$ can be taken to be $c=1$ from the definition of uniform integrability without loss of generality.  From an application of the H\"older inequality (cf., for instance, Theorem 5.4.7 in \cite{G12}), we have that it suffices to show that $Y_{k}^{2a}$ and $W_{k}^{2b}$ are uniformly integrable. However, this holds for the former sequence by the induction hypothesis, and for the latter by the assumptions of the theorem.                                                                
\end{proof}

From [(\ref{coupl1}), Lemma \ref{propcoup}] we have that
\begin{equation}\label{tpbound}
\left(\frac{\til{T}_{r}}{r^{2}}\right)^{p} \leq \left( \sum_{n=1}^{N} \frac{\tau_{n,r}}{r^{2}}\right)^{p}
\end{equation}
almost surely. Further, from [(\ref{coupl2}), Lemma \ref{propcoup}] we have that 
\begin{equation}\label{mbound}
\left(\frac{\til{M}_{t}}{\sqrt{t}}\right)^{p} \leq  \left(\sum_{n=1}^{N} \frac{m_{n,t}}{\sqrt{t}}\right)^{p}
\end{equation}
almost surely. From Lemma \ref{gener},  we hence have that  (\ref{intA}) follows 
from (\ref{tpbound}) due to [(\ref{mom1}); Lemma \ref{lmomEK}], whereas (\ref{intB}) follows from (\ref{mbound}) due to [(\ref{aimm}); Lemma \ref{lmomEK}]. 
\end{proof} 

\begin{remark}\textup{We note that a another proof to Lemma \ref{gener} can be sketched as follows. The H\"older inequality for positive random variables gives that $\left(\sum_{n=1}^{N} X_{n,k}\right)^{p} \leq N^{p/q} \sum_{n=1}^{N} X_{n,k}^{p}$,
$\frac{1}{p}+\frac{1}{q} = 1$, from which, by invoking (Theorem 5.4.6 in \cite{G12}), another proof can be shown.}
\end{remark}

\subsubsection{Proof of Theorem \ref{momeN}}\label{subthmprf}
\begin{proof}[Proof of Theorem \ref{momeN}.] We will first show the following statement we need to use. 
\begin{lemma}\label{TMcon}
\begin{equation}\label{mula}
\frac{\til{T}_{r}^{p}}{r^{2p}}  \xrightarrow{d} \wdh{T}^{p}  \hs \mbox{ as }  r \rightarrow \infty.
\end{equation}
and that 
\begin{equation}\label{mulb}
\frac{1}{t^{p/2}} \til{M}_{t}^{p} \xrightarrow{d} \wdh{M}^{p}  \hs \mbox{ as }  t \rightarrow \infty, 
\end{equation}
\end{lemma}
\begin{proof} Let $C_{N}$ be all continuous $\phi:[0,\infty) \rightarrow \R^{N}$, equip with the usual metric $\rho$, so that $\rho(\phi_{n}, \phi_{0}) \rightarrow 0$, as $n \rightarrow \infty$, if and only if, for any $k$,
\[
\lim_{n \rightarrow \infty} \sup_{0 \leq t \leq k}|\phi_{n}(t) - \phi_{0}(t)| = 0.  
\]
Let ${\bf{W}}^{(n)}(s, \omega) = \frac{1}{\sqrt{n}}{\bf Z}_{[ns]}(\omega)$, let also ${\bf W}(t)$ be $N$-dimensional Brownian motion covariance matrix $\Sigma = \frac{1}{N} I$, and further let $\Psi: C_{N} \rightarrow \R$ be any function which is continuous. From, for instance, [Theorem 3.5.1, \cite{LL10}], we have that in $C_{N}$,
\begin{equation}\label{bro}
\Psi({\bf{W}}^{(n)}) \Rightarrow \Psi({\bf W}), \mbox{ as } n \rightarrow \infty.  
\end{equation}

Let $\Psi_{1}({\bf x}) = \inf\{s: |{\bf x}(s)| >1 \}$ and let $M_{n} = \inf\{ t: | {\bf Z}_{t} | > \sqrt{n}\}$. Since $\Psi_{1}$ is a.s.\ continuous and, by linear interpolation, 
\[
\inf\{ ns: |{\bf Z}_{[ns]}| \geq \sqrt{n} \} = \inf\{ s:  |{\bf Z}_{s}| \geq \sqrt{n}\}, 
\]
for all $n$, (\ref{bro}) gives
\[
\frac{M_{n}}{n} \xrightarrow{d} \wdh{T},  \mbox{ as } n \rightarrow \infty, 
\]
so that (\ref{mula}) follows by an application of the continuous mapping theorem.

Let $\Psi_{2}({\bf x}) = \sup_{0 \leq s \leq 1} |{\bf x}(s)|$. Since $\Psi_{2}$ is a.s.\ continuous and by linear interpolation we have that
\begin{equation*}\label{interp}
\sup_{0 \leq s \leq 1} \frac{|{\bf Z}_{[n s]}|}{\sqrt{n}}  =  \sup_{0 \leq j  \leq n} \frac{| { \bf Z}_{j} |}{\sqrt{n}},  
\end{equation*}
for all $n$, (\ref{bro}) gives
\begin{equation*}\label{eq:}
 \sup_{0 \leq j  \leq n} \frac{| { \bf Z}_{j} |}{\sqrt{n}} \xrightarrow{d}  \wdh{M}.
\end{equation*}
which yields (\ref{mulb}) by an application of the continuous mapping theorem, and hence, the proof is complete. 
\end{proof}

\noindent To complete the proof, note that the limiting distributions associated to  (\ref{mula}) and  (\ref{mulb}) are given by [\ref{Ftil}, Proposition \ref{corBM}] and by [\ref{Gtil}, Proposition \ref{corBM}] respectively. Hence, by a general result, see for instance [Theorem 5.5.9, \cite{G12}], we have that (\ref{eqA}) follows from (\ref{mula}) combined with [(\ref{intA}), Proposition \ref{uin}], and that (\ref{eqB}) follows from (\ref{mulb}) combined with [(\ref{intB}), Proposition \ref{uin}]. Thus, the proof is complete.
\end{proof}

\noindent \textbf{Acknowledgments.} I wish to thank two anonymous referees for detailed suggestions pertaining to presentation and misprints. The work of the author is currently supported financially by {\small PNPD/CAPES} postdoc program.

\textsc{ \\
\noindent Instituto de Matem\' atica e Estat\' istica,\\
Universidade de S\~ ao Paulo\\
Rua do Mat\~ ao, 1010\\
CEP 05508-900- S\~ ao Paulo\\
Brasil}

\end{document}